%% file: Paper.tex
\documentclass[abstracton]{scrartcl}

\usepackage[T1]{fontenc}
\usepackage[utf8]{inputenc}
\usepackage[english]{babel}

\usepackage{amssymb,amsmath,amsthm,enumerate,mathtools}

\usepackage[linkcolor=black, citecolor=black,filecolor=black,menucolor=black,urlcolor=black, colorlinks,breaklinks]{hyperref}
\usepackage{todonotes}

\DeclareMathOperator{\rang}{ran}
\DeclareMathOperator{\spann}{span}

\begin{document}

\input{kuerzel}

\title{On the Convergence Theorem for the Regularized Functional Matching Pursuit (RFMP) Algorithm}
\author{\footnotesize
\begin{minipage}{0.45\textwidth}
Volker Michel
\\ Geomathematics Group\\ Department of Mathematics\\ University of Siegen\\ Germany\\ \url{michel@mathematik.uni-siegen.de}
\end{minipage}
\and 
\begin{minipage}{0.45\textwidth}\footnotesize
Sarah Orzlowski
\\ Geomathematics Group\\ Department of Mathematics\\ University of Siegen\\ Germany\\ \url{orzlowski@mathematik.uni-siegen.de}
\end{minipage}
}
\date{\today}
\maketitle

\begin{abstract}
The RFMP is an iterative regularization method for a class of linear inverse problems. It has proved to be applicable to problems which occur, for example, in the geosciences. In the early publications \cite{Fischer_Diss,FisMic2012}, it was shown that the iteration converges in the unregularized case to an exact solution. In \cite{Michel_RFMP} and \cite{MicTel2014}, it was later shown (for two different scenarios) that the iteration also converges in the regularized case, where the limit of the iteration is the solution of the Tikhonov-regularized normal equation. However, the condition of these convergence proofs cannot be satisfied and, therefore, has to be weakened, as it was pointed out for the convergence theorem of the related iterated Regularized Orthogonal Functional Matching Pursuit (ROFMP) algorithm in \cite{MicTel2016}. Moreover, the convergence proof in \cite{Michel_RFMP} contained a minor error. For these reasons, we reformulate here the convergence theorem for the RFMP and its proof. We also use this opportunity to extend the algorithm for an arbitrary infinite-dimensional separable Hilbert space setting. In addition, we particularly elaborate the cases of non-injective and non-surjective operators.
\end{abstract}

\section{Summary of the RFMP}
The RFMP is an algorithm for the regularization of inverse problems of the following type.
\begin{probl}\label{Problem} 
  Let $\cH$ be a separable and infinite-dimensional Hilbert space (of functions), $\ell\in\NN$ be the dimension of the data space, $y\in\R^{\ell}$ be a given (data) vector and $\F\colon\cH\to\R^{\ell}$ be a given linear and continuous operator. The problem is to find (a function) $F\in\cH$ such that
  \[
  \F F=y\, .
  \]
\end{probl}
%We denote the components of $\F$ by $\F^k$, that is, $\F F=(\F^1 F,\dots,\F^{\ell} F)^\mathrm{T}$.
The RFMP tries to iteratively construct a sequence $(F_n)_n\subset\cH$ of approximations to the solution $F$.
%
% Algorithm
%
\begin{algo}[Regularized Functional Matching Pursuit, RFMP]\label{RFMP}
  Let an initial approximation $F_0\in\cH$ (e.g.\ $F_0=0$) be given. Moreover, choose a dictionary $\cD\subset\cH\setminus\{0\}$ of (possibly useful) trial functions.
  \begin{enumerate}
    \item Initialize the step number to $n\coloneqq0$ and the residual to $R^0\coloneqq y-\F F_0$ and choose a regularization parameter $\la\in\R_0^+$.
    \item\label{Alg_Step2} Determine
    \begin{align}
    d_{n+1} &\coloneqq \operatorname*{\arg\,\max}_{d\in\cD} \frac{\left(\llan R^n,\F d\rran_{\R^{\ell}} - \lambda \llan F_n, d\rran_\cH\right)^2}{\left\| \F d\right\|_{\R^{\ell}}^2 +\lambda \left\| d\right\|_\cH^2}\, , \label{Alg_Wahl_d}\\
    \alpha_{n+1} &\coloneqq \frac{\llan R^n,\F d_{n+1}\rran_{\R^{\ell}} - \lambda \llan F_n, d_{n+1}\rran_\cH}{\left\| \F d_{n+1}\right\|_{\R^{\ell}}^2 +\lambda \left\| d_{n+1}\right\|_\cH^2}\label{Alg_Wahl_alpha}
    \end{align}
    and set $F_{n+1}\coloneqq F_n+\alpha_{n+1}d_{n+1}$ and $R^{n+1}\coloneqq R^n-\alpha_{n+1} \F d_{n+1}$.
    \item Increase $n$ by 1 and go to step \ref{Alg_Step2}.
  \end{enumerate}
\end{algo}
In practice, the algorithm will be stopped by an appropriate criterion (see e.g.\ \cite{Michel_RFMP}). Since we are interested in a convergence theorem, we neglect this aspect here.
%
% Convergence
%
\section{The Convergence Theorem}

Several properties can be proved for the RFMP. We summarize here only one result which we will need (see \cite[Eq.\ (2) and Theorem 1]{Michel_RFMP}) for the convergence proof. Note that we used an $\rL^2$-space in the earlier publications instead of a general Hilbert space $\cH$. The proofs are, however, easily transferable to the general case.
\begin{lemma}\label{Convlemma}
  The sequences $(F_n)_n\subset\cH$ and $(R^n)_n\subset\R^{\ell}$ of the RFMP satisfy
  \begin{align}
  \MoveEqLeft{\left\|R^{n}\right\|_{\R^{\ell}}^2 +\lambda\left\|F_{n}\right\|_\cH^2 }\nonumber\\
  &= \left\|R^{n-1}\right\|_{\R^{\ell}}^2 +\lambda \left\|F_{n-1}\right\|_\cH^2 -\frac{\left(\llan R^{n-1},\F d_{n}\rran_{\R^{\ell}} - \lambda \llan F_{n-1}, d_{n}\rran_\cH\right)^2}{\left\| \F d_{n}\right\|_{\R^{\ell}}^2 +\lambda \left\| d_{n}\right\|_\cH^2}, \label{Gl2}
  \end{align}
  $n \in \NN$, such that the sequence $(\|R^n\|_{\R^{\ell}}^2 +\lambda \|F_n\|_\cH^2)_n$ is monotonically decreasing and convergent.
\end{lemma}
The following theorem improves \cite[Theorem 3.5]{Fischer_Diss}, \cite[Theorem 4.5]{FisMic2012}, \cite[Theorem 2]{Michel_RFMP} and \cite[Theorem 6.3]{MicTel2014}.
%
% Theorem
%

\begin{them}[Convergence Theorem]\label{ConvTh}
  Let the setting of Problem \ref{Problem} be given and let the dictionary $\cD\subset\cH\setminus\{0\}$ satisfy the following properties:
  \begin{enumerate}
    \item\label{ConvThcond1} \lq semi-frame condition\rq: There exists a constant $c>0$ and an integer $M\in\NN$ such that, for all expansions $H=\sum_{k=1}^\infty \beta_k d_k$ with $\beta_k\in\R$ and $d_k\in\cD$, where the $d_k$ are not necessarily pairwise distinct but $|\{j\in\NN \,|\, d_j=d_k\}|\leq M$ for each $k\in\NN$, the following inequality is valid:
        \[
        c\|H\|_\cH^2 \leq \sum_{k=1}^\infty \beta_k^2\, .
        \]
    \item\label{ConvThcond2} $C_1\coloneqq\inf_{d\in\cD} (\|\F d\|_{\R^{\ell}}^2+\lambda\|d\|_\cH^2)>0$.%$C_1\coloneqq\inf_{d\in\cD} \|\F d\|_{\R^{\ell}}>0$ and $C_2\coloneqq \inf_{d\in\cD} \|d\|_\cH >0$.
  \end{enumerate}
  If the sequence $(F_n)_n$ is produced by the RFMP and no dictionary element is chosen more than $M$ times, then $(F_n)_n$ converges in $\cH$ to $F_\infty\coloneqq\sum_{k=1}^\infty \alpha_k d_k + F_0\in\cH$. Moreover, the following holds true:
  \begin{enumerate}[(a)]
    \item\label{ConvThconda} If $\cD$ is a spanning set for $\cH$ (i.e.\ $\overline{\spann \cD}^{\|\cdot\|_\cH}=\cH$), $\cD$ is bounded (i.e.\ $C_2\coloneqq\sup_{d\in\cD} \|d\|_\cH<+\infty$), and $\lambda\in\R_0^+$ is an arbitrary parameter, then $F_\infty$ solves the Tikhonov-regularized normal equation
      \[
      (\F^\ast \F +\lambda \Id)F_\infty =\F^\ast y\, ,
      \]
      where $\F^\ast$ is the adjoint operator corresponding to $\F$ and $\Id$ is the identity operator on $\cH$. This also yields that
        \[
        \left\| y-\F F_\infty \right\|_{\R^{\ell}}^2 + \lambda \left\| F_\infty\right\|_\cH^2 = \min_{F\in\cH}  \left( \left\| y-\F F \right\|_{\R^{\ell}}^2 + \lambda \left\| F\right\|_\cH^2\right)\, ,
        \]
        where the minimizer is unique, if $\lambda>0$.
    \item\label{ConvThcondb} If $\F\cD$ is a spanning set for $\rang\F \subset \R^{\ell}$ (i.e.\ $\spann \{\F d\, |\, d\in\cD \} = \rang \F$) and $\lambda=0$ (no regularization), then $F_\infty$ solves $\F F_\infty = \mathcal{P}_{\rang \F\,} y$, where $\mathcal{P}_{\rang \F}$ is the orthogonal projection onto $\rang \F$.
  \end{enumerate}
\end{them}
\begin{proof}
  With (\ref{Alg_Wahl_alpha}), (\ref{Gl2}), condition \ref{ConvThcond2} of the current theorem and Lemma \ref{Convlemma}, we obtain
  \begin{align*}
  \sum_{k=n}^\infty \alpha_k^2 &=\sum_{k=n}^\infty \frac{ \left\|R^{k-1}\right\|_{\R^{\ell}}^2 + \lambda \left\|F_{k-1}\right\|_\cH^2 - \left( \left\|R^k\right\|_{\R^{\ell}}^2 +\lambda \left\|F_k\right\|_\cH^2\right)}{\left\|\F d_k \right\|_{\R^{\ell}}^2 + \lambda \left\| d_k\right\|_\cH^2}
    \nonumber\\
  &\leq \frac{1}{C_1}\, \sum_{k=n}^\infty \left[ \left\|R^{k-1}\right\|_{\R^{\ell}}^2 + \lambda \left\|F_{k-1}\right\|_\cH^2 - \left( \left\|R^k\right\|_{\R^{\ell}}^2 +\lambda \left\|F_k\right\|_\cH^2\right)\right] \nonumber\\
  &= \frac{1}{C_1} \left[ \left\|R^{n-1}\right\|_{\R^{\ell}}^2 +\lambda\left\|F_{n-1}\right\|_\cH^2 - \lim_{k\to\infty} \left( \left\|R^k\right\|_{\R^{\ell}}^2 +\lambda \left\|F_k\right\|_\cH^2\right)\right]\, .%\label{Gl4}
  \end{align*}
  Consequently, $\lim_{n\to\infty}\sum_{k=n}^\infty \alpha_k^2=0$ and, hence, $\lim_{n\to\infty}\alpha_n=0$. We can define $F_\infty\coloneqq\sum_{k=1}^\infty \alpha_k d_k + F_0$, which is an element of $\cH$ due to the semi-frame condition and the previous estimate. Indeed, $(F_n)_n$ converges to $F_\infty$ in $\cH$ (in the strong sense), also due to the semi-frame condition, which we can see as follows:
  \[
  \lim_{n\to\infty} \left\|F_\infty-F_{n-1}\right\|_\cH^2 = \lim_{n\to\infty} \left\|\sum_{k=n}^\infty\alpha_k d_k \right\|_\cH^2 \leq \frac{1}{c} \, \lim_{n\to\infty} \sum_{k=n}^\infty\alpha_k^2 = 0\, .
  \]
  Since $\F$ is continuous, also $(\F F_n)_n$, $(R^n)_n=(y-\F F_n)_n$ and $(\F^\ast R^n)_n$ must converge (strongly).
  %As a consequence, the sequence of summands must converge to $0$, i.e., with (\ref{Alg_Wahl_alpha}), we get
  %\[
  %\alpha_{n+1}^2 = \left(\frac{\llan R^n,\F d_{n+1}\rran_{\R^{\ell}} - \lambda \llan F_n, d_{n+1}\rran_\cH}{\left\| \F d_{n+1}\right\|_{\R^{\ell}}^2 +\lambda \left\| d_{n+1}\right\|_\cH^2}\right)^2 \longrightarrow 0 \quad \mbox{as } n\to\infty\, .
  %\]

  Due to the continuity of $\F$, the operator norm $\|\F\|_\cL \coloneqq  \sup_{F\in\cH,\, F\neq 0} \frac{\|\F F\|_{\R^{\ell}}}{\|F\|_\cH}$ is finite. We use this together with the boundedness of the dictionary and (\ref{Alg_Wahl_d}) and get, for all $d\in\cD$, the estimate
  \begin{align}\label{eq:EstimateAlpha}
  \alpha_{n+1}^2
  &\geq \frac{1}{(\|\F\|_{\cL}^2 +\lambda)\left\|d_{n+1}\right\|_\cH^2}\, \frac{\left(\llan R^n,\F d_{n+1}\rran_{\R^{\ell}} - \lambda \llan F_n, d_{n+1}\rran_\cH\right)^2}{\left\| \F d_{n+1}\right\|_{\R^{\ell}}^2 +\lambda \left\| d_{n+1}\right\|_\cH^2}\nonumber \notag\\
  &\geq \frac{1}{(\|\F\|_{\cL}^2 +\lambda)C_2^2}\, \frac{\left(\llan R^n,\F d\rran_{\R^{\ell}} - \lambda \llan F_n, d\rran_\cH\right)^2}{\left\| \F d\right\|_{\R^{\ell}}^2 +\lambda \left\| d\right\|_\cH^2} \, .
  \end{align}
  Let us now concentrate on case (\ref{ConvThconda}). Since $\lim_{n\to\infty}\alpha_n=0$, an immediate consequence of the estimate in \eqref{eq:EstimateAlpha} is
  \begin{equation}
  \llan R^n,\F d\rran_{\R^{\ell}} - \lambda \llan F_n, d\rran_\cH = \llan \F^\ast R^n-\lambda F_n,d\rran_\cH \longrightarrow 0 \quad \mbox{as } n\to\infty\label{weakconv}
  \end{equation}
  for all $d\in\cD$.
  Due to the bilinearity of the inner product and the algebraic limit theorem, we also have
  \[
  \llan \F^\ast R^n -\lambda F_n, d\rran_\cH \longrightarrow 0 \quad \mbox{as } n\to\infty
  \]
  for all $d\in\spann\cD$.
  As we derived above, $(\F^\ast R^n - \lambda F_n)_n$ is a strongly convergent and, thus, bounded, sequence. Now let $d\in\cH$ be arbitrary. Due to the first condition in part \eqref{ConvThconda}, there exists a sequence $(\tilde{d}_m)_m\subset\spann\cD$ such that $\|\tilde{d}_m-d\|_\cH\to 0$ as $m\to\infty$. Then the Cauchy-Schwarz inequality yields
  \begin{align*}
    \left| \llan \F^\ast R^n -\lambda F_n, \tilde{d}_m-d\rran_\cH \right| &\leq \left\|\F^\ast R^n - \lambda F_n\right\|_\cH \, \left\|\tilde{d}_m-d\right\|_\cH\\
    &\leq \sup_{n\in\NN_0} \left\|\F^\ast R^n - \lambda F_n\right\|_\cH \, \left\|\tilde{d}_m-d\right\|_\cH\\
    &\rightarrow 0 \quad \text{as }m\to\infty\,.
  \end{align*}
  Since this convergence for $m\to\infty$ is uniform with respect to $n$, we get, by applying the Moore-Osgood double limit theorem, the identity
  \begin{align*}
    \lim_{n\to\infty} \llan \F^\ast R^n -\lambda F_n, d\rran_\cH &= \lim_{n\to\infty}\lim_{m\to\infty} \llan \F^\ast R^n -\lambda F_n, \tilde{d}_m\rran_\cH\\
    &= \lim_{m\to\infty}\lim_{n\to\infty} \llan \F^\ast R^n -\lambda F_n, \tilde{d}_m\rran_\cH\\
    &= 0\, .
  \end{align*}
  This shows that $(\F^\ast R^n)_n$ weakly converges to $\lambda F_\infty$ (and, due to the considerations above, also strongly). Consequently, since $\F^\ast R^n = \F^\ast y -\F^\ast \F F_n$, we obtain, using again the continuity of $\F$, that
  \[
  \F^\ast y - \F^\ast \F F_\infty = \lambda F_\infty\, ,
  \]
  which is equivalent to
  \begin{equation}
    \F^\ast y = \left(\F^\ast\F +\lambda \Id\right)F_\infty\, .\label{Normgl}
  \end{equation}

  It is a basic result of Tikhonov regularization (see e.g.\ \cite[p.\ 89]{Louis_Buch}) that every solution of \eqref{Normgl} minimizes
  \begin{equation}
    \|y-\F F\|_{\R^{\ell}}^2 + \lambda\|F\|_\cH^2
    %\nonumber\\
    %&=& \|y\|_{\R^{\ell}}^2 - 2\lan y, \F F\ran_{\R^{\ell}} +\lan \F F,\F F\ran_{\R^{\ell}} +\lambda \lan F,F\ran_\cH\nonumber\\
    %&=& \|y\|_{\R^{\ell}}^2 - 2 \llan\F^\ast y,F\rran_\cH + \llan \left( \F^\ast\F+\lambda \Id \right) F, F\rran_\cH \nonumber\\
    %&=& \|y\|_{\R^{\ell}}^2 - 2 \llan\F^\ast y,F\rran_\cH + \llan \left( \F^\ast\F+\lambda \Id \right) \left(F-F_\infty\right), F-F_\infty\rran_\cH\nonumber\\
    %&& {} + 2 \llan \left(\F^\ast \F +\lambda \Id \right) F_\infty, F\rran_\cH - \llan \left( \F^\ast\F+\lambda \Id \right) F_\infty, F_\infty\rran_\cH \nonumber\\
    = \|y\|_{\R^{\ell}}^2 + \llan \left( \F^\ast\F+\lambda \Id \right) \left(F-F_\infty\right), F-F_\infty\rran_\cH - \llan\F^\ast y, F_\infty\rran_\cH \,.\label{ConvThBeweqeight}
  \end{equation}
  If $\lambda>0$, then \eqref{Normgl} and the minimization of \eqref{ConvThBeweqeight} both are uniquely solvable by
  \[
  F_\infty= (\F^\ast\F+\lambda \Id)^{-1} \F^\ast y\, .
  \]
  %For the remaining proof of part \eqref{ConvThcondb} of the theorem, we refer to the previous versions of the convergence proofs (see e.g.\ \cite{Michel_RFMP}).
  %is uniquely determined, if $\lambda>0$.\\
  %In the unregularized case (case (\ref{ConvThcondb}), i.e., $\lambda=0$), Theorem \ref{ThProp1} yields that $(\|R^n\|_{\R^{\ell}})_n$ converges. Hence, there exists a convergent subsequence $(R^{n_j})_j$ with limit $R^\infty\in\R^{\ell}$. Moreover, (\ref{Alg_Wahl_d}), (\ref{Gl3}), and the convergence of the series in (\ref{Gl4}) yield that, for all $d\in\cD$,
  %\[
  %0\leq\frac{\llan R^\infty,\F d\rran_{\R^{\ell}}^2}{\|\F d\|_{\R^{\ell}}^2} = \lim_{j\to\infty} \frac{\llan R^{n_j},\F d\rran_{\R^{\ell}}^2}{\|\F d\|_{\R^{\ell}}^2}
  %\leq \lim_{j\to\infty} \frac{\llan R^{n_j},\F d_{n_j+1}\rran_{\R^{\ell}}^2}{\|\F d_{n_j+1}\|_{\R^{\ell}}^2} = 0\, .
  %\]
  %Due to the requirement that $\spann\{\F d\,|\, d\in\cD\}=\R^{\ell}$, we get $R^\infty =0$. Since $(\|R^n\|_{\R^{\ell}})_n$ is monotonically decreasing and every convergent subsequence converges to 0, the sequence $(\|R^n\|_{\R^{\ell}})$ converges to $0$ and, thus, $(R^n)_n$ converges to $0\in\R^{\ell}$. Finally, we use the continuity of $\F$ and conclude that
  %\[
  %\F F_\infty =\lim_{n\to\infty} \F F_n = \lim_{n\to\infty} \left(y-R^n\right) = y\, .
  %\]
%
  For the remaining proof of part \eqref{ConvThcondb} of the theorem, we observe again the estimate in \eqref{eq:EstimateAlpha} with $\lambda = 0$
  \begin{equation*}
    0 \leq \frac{1}{(C_2 \|\F\|_{\cL} )^2 }\, \frac{\llan R^n,\F d\rran_{\R^{\ell}}^2}{\left\| \F d\right\|_{\R^{\ell}}^2} \leq \alpha_{n+1}^2 \longrightarrow 0 \text{ as } n \to \infty\, . 
  \end{equation*}
  With the sandwich theorem, we directly obtain $\lim_{n\to\infty}\llan R^n,\F d\rran_{\R^{\ell}} = 0$ for all $d \in \cD$. Since $\F\mathcal{D}$ is a spanning set for $\rang \F$, which is closed since $\F$ is a finite rank operator, we obtain for all $f \in \rang \F$
  \begin{equation*}
    0 =  \lim_{n\to\infty}\llan R^n, f \rran_{\R^{\ell}} =  \lim_{n\to\infty}\llan \mathcal{P}_{\rang \F\,} R^n, f \rran_{\R^{\ell}} =  \llan \mathcal{P}_{\rang \F\,} R^\infty, f \rran_{\R^{\ell}},
  \end{equation*}
  note that the orthogonal projection $\mathcal{P}_{\rang \F\,}$ is continuous. Thus, $(\mathcal{P}_{\rang \F\,} R^n)_{n}$ converges weakly to zero. In addition, due to the continuity of $\F$, the sequence $(R^n)_{n}$ converges strongly, that is, $R^\infty = \lim_{n\to \infty} R^n$. Due to the uniqueness of the limit and the continuity of the orthogonal projection, we obtain $\mathcal{P}_{\rang \F\,} R^\infty = 0$. Eventually, we get
  \begin{equation*}
    \F F_\infty = \mathcal{P}_{\rang \F\,} (\F F_\infty) = \lim_{n\to\infty} \mathcal{P}_{\rang \F\,} (\F F_n) = \lim_{n\to\infty} \mathcal{P}_{\rang \F\,} ( y - R^n) = \mathcal{P}_{\rang \F\,} y\, . \tag*{\qedhere}
  \end{equation*}
\end{proof}

Note that in the case of a surjective operator $\F$, the statement in part \eqref{ConvThcondb} of the latter theorem coincides with the previous versions in  \cite[Theorem 3.5]{Fischer_Diss}, \cite[Theorem 4.5]{FisMic2012}, \cite[Theorem 2]{Michel_RFMP} and \cite[Theorem 6.3]{MicTel2014}.

\begin{coll}
  If the condition in item (\ref{ConvThcondb}) from the previous theorem is replaced by
  \begin{enumerate}
  \item[(b)]  If $\F\cD$ is a spanning set for the closed set $\mathbb{G} \subset \rang \F \subset \R^{\ell}$ and $\lambda=0$, 
  \end{enumerate}
  then $F_\infty$ solves $\F F_\infty = \mathcal{P}_{\mathbb{G}} y$, where $\mathcal{P}_{\mathbb{G}}$ is the orthogonal projection onto $\mathbb{G}$.
\end{coll}
\begin{proof}
  In analogy to the previous proof, we directly obtain $\lim_{n\to\infty}\llan R^n,\F d\rran_{\R^{\ell}} = 0$ for all $d \in \cD$. Since $\mathbb{G}$ is closed and $\rang \F$ is a spanning set for $\mathbb{G}$, we get for all $g \in \mathbb{G}$
  \begin{align*}
    0 &= \lim_{n\to\infty}\llan R^n,g\rran_{\R^{\ell}} = \lim_{n\to\infty} \left(\llan \mathcal{P}_{\mathbb{G}} R^n,g\rran_{\R^{\ell}} + \llan  \mathcal{P}_{\mathbb{G}^\perp} R^n,g\rran_{\R^{\ell}}\right) \\
    &= \lim_{n\to\infty} \llan \mathcal{P}_{\mathbb{G}} R^n,g\rran_{\R^{\ell}}.
  \end{align*}
  Thus, $ \mathcal{P}_{\mathbb{G}} R^n$ converges weakly to zero. Hence the continuity of $\F$ and the uniqueness of the limits yields
  \begin{align*}
   0 = \mathcal{P}_{\mathbb{G}} (y- \F F_\infty) = \mathcal{P}_{\mathbb{G}} y - \F F_\infty,
  \end{align*}
  since $\mathbb{G}$ is closed and $\mathcal{F}\mathcal{D}$ spans $\mathbb{G}$.
\end{proof}

In \cite[Lem. 4.2.5]{Tel2014}, it was proved that $F_\infty \in (\ker \F)^\perp$. Hence, the condition $\overline{\spann \cD}^{\|\cdot\|_\cH}=\cH$ in the convergence theorem is unnecessarily strong. If more knowledge of the operator $\F$ is available, for instance, the singular value decomposition $(\sigma_j; x_j, y_j)$, the condition for the dictionary in case \eqref{ConvThconda} of Theorem \ref{ConvTh} can be weakened.

\begin{them}
Let $F_\infty$ denote the unique solution of the Tikhonov-regularized normal equation $(\F^\ast\F+\lambda \Id) F_\infty= \F^\ast y$, where $\lambda > 0$. Let $(\sigma_j; x_j, y_j)$ denote the singular system of $\F$ and let the set $V$ be defined by $V \coloneqq \overline{\bigcup_{j \in J} \{x_j\}}$, where $J \subset \NN$ is a countable index set.
%  \begin{enumerate}[(a)]
    %\item 
    %If $\cD$ is a spanning set for the closed set $V \subset \cH$, $\cD$ is bounded, and $\lambda\in\R^+$ is an arbitrary parameter, 
    If the conditions of Theorem \ref{ConvTh} with the case \eqref{ConvThconda} are satisfied, except that $\cD$ is (only) a spanning set for $V$,
    then the solution $F_{\infty, V}$ produced by the RFMP solves 
      \[
      F_{\infty, V} = \mathcal{P}_{V} F_\infty\, ,
      \]
      where $\mathcal{P}_{V}$ is the orthogonal projection onto $V$.
%  \end{enumerate}
\end{them}
\begin{proof}
  $(V, \lan \cdot, \cdot \ran_{\cH})$ is a Hilbert space, since $V \subset \cH$ is closed. The operator $\F_V \coloneqq \F \mathcal{P}_V$ is a bounded operator $\F_V \colon \cH \to \R^\ell$, and hence, its restriction $\F_V|_V \colon V \to \R^\ell$ is also bounded, where $\F_V|_V = \F|_V$. We can apply Theorem \ref{ConvTh} to this setting and obtain the solution $F_{\infty,V} \in V$ produced by the RFMP, which solves the Tikhonov-regularized normal equation in $V$, that is,
  \begin{equation*}
     (\F_V^\ast \F_V +\lambda \Id_V)F_{\infty,V} =\F_V^\ast y\, .
  \end{equation*}

    In order to prove that $F_{\infty,V}$ is the best approximation of $F_{\infty}$ in $V$, it remains to show that $F_{\infty,V} = \mathcal{P}_V F_{\infty}$. For this purpose, we study the singular system $(\sigma_j; x_j, y_j)$ of $\F$, which exists due to the compactness of $\F$. Due to the construction of $V$, we obtain, for each $j\in\NN$, that $x_j$ is either in $V$ or in $V^\perp$. Hence, $\F^\ast \F$ and $\mathcal{P}_V$ commute, that is,
    \begin{align*}
        \mathcal{P}_V \F^\ast \F F = \sum_{\substack{j=1, \\ x_j \in V}}^\infty \sigma_j^2  \llan F, x_j \rran x_j = \F^\ast \F \mathcal{P}_V F\, \qquad \text{for all } F \in \mathcal{H}.
    \end{align*}
    Due to $\F_V = \F \mathcal{P}_V$, we directly obtain $\F^\ast_V = \mathcal{P}_V \F^\ast$. For $F_{\infty,V}$, we get
    \begin{align*}
        (\F^\ast \F + \lambda \Id) F_{\infty, V} %&= (\F^\ast \F \mathcal{P}_V + \lambda \mathcal{P}_V) F_{\infty, V} \\
        &= (\F^\ast \F \mathcal{P}_V^2 + \lambda \mathcal{P}_V) F_{\infty, V} 
        = (\mathcal{P}_V \F^\ast \F \mathcal{P}_V  + \lambda \mathcal{P}_V ) F_{\infty, V}\\
        &= ( \F^\ast_V \F_V  + \lambda \mathcal{P}_V ) F_{\infty, V}
        = \F^\ast_V y = \mathcal{P}_V \F^\ast y \\
        &= \mathcal{P}_V ( (\F^\ast \F + \lambda \Id) F_\infty) 
        =   (\F^\ast \F + \lambda \Id) \mathcal{P}_V F_\infty.
    \end{align*}
    Since $\F^\ast \F + \lambda \Id$ is one-to-one, we eventually get $F_{\infty, V} = \mathcal{P}_V F_\infty$.
    % \begin{align*}
    %   (\F^\ast \F + \lambda \Id)^{-1} F &= \sum_{j=1}^\infty \frac{1}{\sigma_j^2 + \lambda } \llan F, x_j \rran x_j\, , \qquad \text{for all } F \in \cH\, , \\
    %   F_\infty = (\F^\ast \F + \lambda \Id)^{-1}\F^\ast y &= \sum_{j=1}^\infty \frac{\sigma_j}{\sigma_j^2 + \lambda } \llan y, y_j \rran x_j\, .
    % \end{align*}
    % For the projection, we obtain
    % \begin{align*}
    %   \mathcal{P}_V F &= \sum_{\substack{j=1 \\ x_j \in V}}^\infty  \llan F, x_j \rran x_j\, , \qquad \text{for all } F \in \cH\, , \\
    %   \mathcal{P}_V F_\infty &= \sum_{\substack{j=1 \\ x_j \in V}}^\infty  \frac{\sigma_j}{\sigma_j^2 + \lambda } \llan y, y_j \rran\mathcal{P}_V  x_j  = \sum_{\substack{j=1 \\ x_j \in V}}^\infty  \frac{\sigma_j}{\sigma_j^2 + \lambda } \llan y, y_j \rran x_j\, .
    % \end{align*}
    % Hence, for the operator $\F_V$, we get
    % \begin{align*}
    %   \F_V F = \F \mathcal{P}_V \left(\sum_{j=1}^\infty  \llan F, x_j \rran x_j\right) %\\
    %   %&= \sum_{j=1}^\infty  \llan F, x_j \rran \F|_V x_j \\
    %   = \sum_{\substack{j=1 \\ x_j \in V}}^\infty \sigma_j \llan F, x_j \rran  x_j \\
    % \end{align*}
    % Eventually, we obtain for the approximation $F_{\infty,V}$
    % \begin{align*}
    % F_{\infty,V} = (\F_V^\ast \F_V + \lambda \mathcal{P}_V)^{-1}\F_V^\ast y &= \sum_{\substack{j=1 \\ x_j \in V}}^\infty \frac{\sigma_j}{\sigma_j^2 + \lambda } \llan y, y_j \rran  x_j = \mathcal{P}_V F_{\infty}\, ,
    % \end{align*}
    % hence, the result $F_{\infty,V}$ produced by the RFMP is the best approximation in $V$ of the unique Tikhonov-regularized minimum norm solution $F_\infty$.
\end{proof}
In the case of a non-injective operator $\F$, we can choose, for example, $V = (\ker \F)^\perp$ and obtain $F_{\infty,(\ker \F)^\perp} = \mathcal{P}_{(\ker \F)^\perp} F_\infty = F_\infty$.

\begin{rem}
  The corrections of the previous versions of the convergence theorem and its proof are as follows:
  \begin{itemize}
    \item In the semi-frame condition, it is now required that there is a maximum number of recurring choices of the same dictionary element. This maximum number can be arbitrarily large as long as it is finite and universal for all considered expansions. The reason for this limitation, which then also has to be required for the solution $F_\infty$ generated by the RFMP (see the proof of the existence of the limit $F_\infty$ above), is that the semi-frame condition could not be satisfied otherwise, as we pointed out for the convergence proof of the ROFMP in \cite{MicTel2016}: if $c$ is the constant of the semi-frame condition, then an unlimited number of recurring choices of dictionary elements could yield (with $d\in\cD$ arbitrary)
        \[
        c\left(\sum_{k=1}^N \frac{1}{k}\right)^2 \|d\|^2_\cH= c\left\|\sum_{k=1}^N \frac{1}{k}\, d\right\|^2_\cH \leq \sum_{k=1}^N\frac{1}{k^2} \quad \text{for all }N\in\NN\,,
        \]
        which cannot be satisfied for any $c>0$. Certainly, the semi-frame condition now has to be seen even more critically, since it implies limitations for the algorithmic choice of the dictionary elements. There is still a gap between the theory of the method, which uses limits to infinity as justifications for the results, and the practical implementation where, certainly, only finite dictionaries and stopped iterations are used.
    \item The conclusions after \eqref{weakconv} were previously erroneous and were corrected here. As we showed here, the strong convergence of $(\F^\ast R^n-\lambda F_n)_n$ is used to obtain the weak convergence to $0$ in $\cH$ out of \eqref{weakconv}. However, in \cite{Michel_RFMP}, it was only \emph{after} this corresponding part that the strong convergence was proved.
  \end{itemize}
\end{rem}
\enlargethispage{2\baselineskip}
\section*{Acknowledgements}
We gratefully acknowledge the support by the German Research Foundation (DFG), project MI 655/10-1.
%
% References
%

\end{document}

%% file: kuerzel.tex
\newcommand{\R}{{\mathbb{R}}}
\newcommand{\NN}{{\mathbb{N}}}
\newcommand{\la}{\lambda}
\newcommand{\ve}{\varepsilon}
\newcommand{\vp}{\varphi}
\newcommand{\F}{{\mathcal{F}}}
\newcommand{\cD}{{\mathcal{D}}}
\newcommand{\cB}{\mathcal{B}}
\newcommand{\cL}{\mathcal{L}}
\newcommand{\cH}{{\mathcal{H}}}
\newcommand{\dx}{{\,{\mathrm d}x}}
\newcommand{\sn}{\sum_{n=0}^\infty}
\newcommand{\sm}{\sum_{m=0}^\infty}
\newcommand{\sj}{\sum_{j=1}^{2n+1}}
\newcommand{\sk}{\sum_{k=1}^{2m+1}}

\newcommand{\Om}{\Omega}
\newcommand{\rL}{{\mathrm{L}}}
\newcommand{\rC}{{\mathrm{C}}}
\newcommand{\rCe}{{\mathrm{C}^{(1)}}}
\newcommand{\LD}{{\rL^2(D)}}
\newcommand{\LO}{{\rL^2(\Omega)}}
\newcommand{\LB}{{\rL^2(\cB)}}
\newcommand{\lan}{\langle}
\newcommand{\ran}{\rangle}
\newcommand{\llan}{\left\langle}
\newcommand{\rran}{\right\rangle}
\newtheorem{them}{Theorem}
\newtheorem{defi}[them]{Definition}
\newtheorem{coll}[them]{Corollary}
\newtheorem{lemma}[them]{Lemma}
\newtheorem{algo}[them]{Algorithm}
\newtheorem{rem}[them]{Remark}
\newtheorem{probl}[them]{Problem}

\newcommand{\Id}{{\mathcal{I}}}